\documentclass[12pt]{amsart}
\usepackage{graphicx}
\usepackage{amsmath, amsxtra, amssymb, latexsym}
\usepackage{hyperref}
\newtheorem{thm}{Theorem}[section]
\newtheorem{cor}[thm]{Corollary}
\newtheorem{lem}[thm]{Lemma}

\theoremstyle{definition}
\newtheorem{defn}[thm]{Definition}
\theoremstyle{remark}
\newtheorem{rem}[thm]{Remark}

\numberwithin{equation}{section}

\def\XXint#1#2#3{{\setbox0=\hbox{$#1{#2#3}{\int}$}
  \vcenter{\hbox{$#2#3$}}\kern-.5\wd0}}

\begin{document}

\title[Codimension four taut Riemannian foliation]
{Direct sum for basic cohomology of codimension four taut Riemannian foliation}

\author{Jiuru Zhou}
\address{School of Mathematical Science \\
Yangzhou University \\
Yangzhou,  Jiangsu 225002, China}
\email{zhoujiuru@yzu.edu.cn}




\date{\today}

\begin{abstract}
We discuss the decomposition of degree two basic cohomology  for codimension four taut Riemannian foliation according to the holonomy
invariant transversal almost complex structure $J$, and show that $J$ is $C^{\infty}$ pure and full.
  In addition, we obtain an estimate of the dimension of basic $J$-anti-invariant subgroup. These are the foliated version for the corresponding results of T. Draghici et al. (Int. Math. Res. Not. IMRN 1:1¨C17, 2010).
\end{abstract}


\maketitle


\section{Introduction}

In order to study S.K. Donaldson's tamed to compatible question \cite{Don}, T.-J. Li and W. Zhang \cite{LZ} defined two subgroups $H^{+}_J(M),~H^{-}_J(M)$ of the real degree 2 de Rham cohomology group $H^{2}(M;\mathbb{R})$ for a compact almost complex manifold $(M,J)$. They are the sets of cohomology classes which can be represented by $J$-invariant and $J$-antiinvariant real 2-forms respectively. Later, T. Draghici, T.-J. Li and W. Zhang showed in \cite{DLZ} that in dimension four
 $$
  H^{2}(M;\mathbb{R})=H^{+}_J(M)\oplus H^{-}_J(M),
 $$
and they call such almost complex structure $J$ to be $C^{\infty}$-pure and full. This is specifical for four dimension, since A. Fino and A. Tomassini's Example 3.3 in \cite{FT} gives a six dimensional almost complex manifold $(M,J)$ with $J$ being not $C^{\infty}$-pure, and higher dimensional non-$C^{\infty}$-pure examples can be obtained by producting it with another almost complex manifold (see Remark 2.7 in \cite{DLZ}). It becomes nature to ask when will the almost complex structure be $C^{\infty}$-pure and full on higher dimension. Inspired by the work of Y. Kordyukov, M. Lejmi and P. Weber \cite{KLW} who generalized the Seiberg-Witten invariants onto codimension four Riemannian foliations, this article lays the groundwork for the case in which the higher dimensional manifold admits a codimension four taut Riemannian foliation $\mathcal{F}$. The main result is Theorem \ref{main} which basically says a transversal almost complex structure $J$ on a codimension four taut Riemannian foliated manifold satisfying $\theta(V)J=0,~\forall~V\in \Gamma T\mathcal{F}$ is $C^{\infty}$-pure and full in the sense of Definition \ref{defn}.

The structure of this article is the following: Section 2 are notions of transverse structures, basic forms, characteristic form and filtrations needed later. We consider the compatibility of transversal almost complex structure with a taut Riemannian metric in section 3. After these preliminaries, basic $J$-(anti)invariant cohomology groups naturally come out and so is $C^{\infty}$-pureness and fullness of $J$. After some lemmas similar to those in \cite{DLZ}, we are able to proof the decomposition of the real degree 2 de Rham cohomology group in section 4. The last section provides bounds on the dimension of $J$-(anti)invariant cohomology groups.

\section{Taut Riemannian Foliation}
Let's first recall some definitions and results in foliations, the below in this section is referred to \cite{Ton}. Let $M$ be a closed oriented smooth manifold of dimension $n=p+q$ endowed with a codimension $q$ foliation $\mathcal{F}$. The integrable subbundle $T\mathcal{F}\subset TM$ is given by vectors tangent to plaques, then we further have the rank $q$ normal bundle defined as the quotient bundle $Q=TM/T\mathcal{F}$ and the projection
 \begin{eqnarray*}
  \pi:~TM  &\rightarrow&  Q\\
       Y   &\mapsto&      \pi(Y)
 \end{eqnarray*}
   denoted by $\overline{Y}=\pi(Y)$.

Define the $\Gamma T\mathcal{F}$-action on $\Gamma Q$ as
 $$
  \theta(V)s=\overline{[V,Y_s]}~~\textrm{for}~~V\in\Gamma T\mathcal{F},~s\in\Gamma Q,
 $$
where $Y_s\in\Gamma TM$ is any choice with $\overline{Y}_s=s$. It can be checked that the definition $\theta(V)s$ is independent of the choice of $Y_s$.

Consider a Riemannian metric $g=g_{T\mathcal{F}}\oplus g_{T\mathcal{F}^{\perp}}$ on $M$ splitting $TM$ orthogonally as $TM=T\mathcal{F}\oplus T\mathcal{F}^{\perp}$, which means there is a bundle map $\sigma : Q \stackrel{\cong}{\rightarrow} T\mathcal{F}^{\perp} \subset TM$ splitting the exact sequence
 $$
  0\rightarrow T\mathcal{F} \rightarrow T M \rightarrow Q \rightarrow 0,
 $$
i.e. satisfying $\pi \circ \sigma$=identity. This induces a metric on $Q$ by $g_{Q}=\sigma^{*} g_{T\mathcal{F}^{\perp}}$, then the splitting map $\sigma:~(Q,g_Q)\rightarrow (T\mathcal{F}^{\perp},g_{T\mathcal{F}^{\perp}})$ is a metric isomorphism.

Suppose $\nabla^M$ is the Levi-Civita connection induced by the Riemannian metric $g$ on $M$. For $s\in \Gamma Q$, define
 $$
  \nabla_{X} s=\left\{\begin{array}{ll}{\pi\left[X, \sigma(s)\right]} & {\text { for } X \in \Gamma T\mathcal{F}} \\ {\pi\left(\nabla_{X}^{M} \sigma(s)\right)} & {\text { for } X \in \Gamma T\mathcal{F}^{\perp}}\end{array}\right.,
 $$
then $\nabla$ is an adapted connection in $Q$, which means $\nabla$ restricting along $T\mathcal{F}$ is the partial Bott connection.

Consider the $Q$-valued bilinear form on $T\mathcal{F}$, i.e. $\alpha : T\mathcal{F} \otimes T\mathcal{F} \rightarrow Q$ given by
 $$
  \alpha(U, V)=\pi\left(\nabla_{U}^{M} V\right) \quad \text { for } U, V \in \Gamma T\mathcal{F}.
 $$

A calculation shows for $s\in \Gamma Q$,
 \begin{equation}\label{bilinear}
  \left(\theta(Y) g_{T\mathcal{F}}\right)(U, V)=-2 g(Y, \alpha(U, V)).
 \end{equation}

The Weingarten map $W(s) : T\mathcal{F} \rightarrow T\mathcal{F}$ is defined by
 $$
  g_{Q}(\alpha(U, V), s)=g(W(s) U, V).
 $$
Then Tr$W\in \Gamma Q^{\ast}$, and it can be extended to a 1-form $\kappa\in \Omega^1(M)$ by setting $\kappa(V)=0$ for $V \in \Gamma T\mathcal{F}$, where we have used the identification $T\mathcal{F}^{\perp} \cong Q$. We call $\kappa$ the mean curvature 1-form of $\mathcal{F}$ on $(M,g)$.

Recall that a Riemannian foliation is a foliation $\mathcal{F}$ with a holonomy invariant transversal metric $g_Q$ on $Q$, i.e.
 $$
  \theta(V)g_Q=0,~\forall~V\in\Gamma T\mathcal{F}.
 $$

The metric $g$ on $(M,\mathcal{F})$ is called bundle-like if the induced metric $g_Q$ is holonomy invariant, i.e., $\theta(V)g_Q=0$ for all $V\in\Gamma T\mathcal{F}$, and a Riemannian foliation $\mathcal{F}$ is called taut if there exists a bundle-like metric for which the mean curvature 1-form $\kappa=0$.

A differential form $\alpha\in \Omega^r(M)$ is basic, if
 $$ i(V)\alpha=0,~\theta(V)\alpha, \forall~V\in\Gamma L. $$
Denote by $\Omega^{\ast}_B=\Omega^{\ast}_B(\mathcal{F})$ the set of all basic forms, and the exterior differential $\mathrm{d}_B=\mathrm{d}|_{\Omega_B}$. By Cartan's magic formula, it can be checked that $(\Omega^{\ast}_B,\mathrm{d}_B)$ forms a sub-complex of the de Rham complex $(\Omega^{\ast},\mathrm{d})$. The corresponding cohomology
 $$H^{\ast}_B(\mathcal{F})=H^{\ast}_B(\mathcal{F};\mathbb{R})$$
is called the basic cohomology of $\mathcal{F}$.

If $T\mathcal{F}$ is oriented, the foliation $\mathcal{F}$ is then said to be tangentially oriented. The $p$-form $\chi_{\mathcal{F}}$ defined by
 $$
  \chi_{\mathcal{F}}\left(Y_{1}, \ldots, Y_{p}\right)=\operatorname{det}\left(g\left(Y_{i}, E_{j}\right)_{i j}\right), \forall~ Y_{1}, \dots, Y_{p} \in \Gamma TM,
 $$
is called the characteristic form of $\mathcal{F}$, where $\{E_1,E_2,\cdots,E_p\}$ is a local oriented orthonormal frame.

Consider the multiplicative filtration of the de Rham complex $\Omega^{\ast}=\Omega^{\ast}(M)$ as follows
 $$
  F^r\Omega^m = \{ \alpha\in\Omega^m~|~i(V_1)\cdots i(V_{m-r+1})\alpha=0~~ \textrm{for}~~ V_1,\cdots,V_{m-r+1}\in \Gamma T\mathcal{F} \}.
 $$

Obviously,
 $$
  F^{0} \Omega^{m}=\Omega^{m} \text { and } \quad F^{m+1} \Omega^{m}=0.
 $$

Furthermore, for the foliation $(M,\mathcal{F})$, we have
 \begin{equation}\label{q+1fil}
  F^{q+1} \Omega^{m}=0 \quad(q=\operatorname{codim} \mathcal{F}).
 \end{equation}

\section{Holonomy invariant transversal almost complex structure}

If the foliation $\mathcal{F}$ is of even codimension, and there exists almost complex structure $J$ on $Q$, i.e. an endomorphism $J : Q \rightarrow Q$ such that $J^{2}=-Id_{Q}$, then extend $J$ onto $TM$ by setting $JX=0$ for $X\in T\mathcal{F}$. Such $J$ is called the transversal almost complex structure.

 \begin{lem}
  For an even codimensional Riemannian foliation $(M,\mathcal{F})$ with a taut Riemannian metric $g=g_{T\mathcal{F}}\oplus g_{T\mathcal{F}^{\perp}}$, if there exists a transversal almost complex structure $J$ satisfying $\theta(V)J=0$ for any $V\in \Gamma T\mathcal{F}$ (we call such $J$ to be holonomy invariant), then the new metric $g_J$ defined by
   $$
    g_J(X,Y)=\left\{\begin{array}{ll}{g_{T\mathcal{F}}(X,Y)} & {\text { for } X,Y \in \Gamma T\mathcal{F}} \\ {g_{T\mathcal{F}^{\perp}}(X,Y)+g_{T\mathcal{F}^{\perp}}(JX,JY)} & {\text { for } X,Y \in \Gamma T\mathcal{F}^{\perp}}\end{array}\right.
   $$
  is also taut.
 \end{lem}

 \begin{proof}
  Since $\theta(V)J=0$ for any $V\in \Gamma T\mathcal{F}$ and $g$ is bundle-like,
   $$
    (\theta(V)g_{J,Q})(s,s')=(\theta(V)g_{Q})(s,s')+(\theta(V)g_{Q})(Js,Js')=0,
   $$
  i.e. $g_J$ is also bundle-like.

For the tautness part, let $e_{1}, \dots, e_{n}$ be an orthonormal basis
of $T_xM$ such that $e_{1}, \ldots, e_{p} \in T\mathcal{F}_{x}$ and $e_{p+1}, \ldots, e_{n} \in T\mathcal{F}_{x}^{\perp}$. Then by (\ref{bilinear}), we have the mean curvature 1-form $\kappa$ for $g$,
 \begin{eqnarray*}
  \kappa(s)_{x}&=&\operatorname{Tr} W(s)_{x}\\
               &=&\sum_{i=1}^{p} g\left(W(s) e_{i}, e_{i}\right)\\
               &=&\sum_{i=1}^{p} g_{Q}\left(\alpha\left(e_{i}, e_{i}\right), s\right)\\
               &=&-\frac{1}{2}\sum_{i=1}^{p}\left(\theta(s) g_{T\mathcal{F}}\right)(e_i, e_i),
 \end{eqnarray*}
which shows that $\kappa$ is independent of $g_Q$.

  We denote by $\kappa_J$ the mean curvature 1-form with respect to $g_J$. Since $g$ is taut, $\kappa$ vanishes, and so is $\kappa_J$, i.e. $g_J$ is also taut.
 \end{proof}

In the sequel, we still denote this $g_J$ by $g$, and call the taut Riemannian metric $g$ compatible with $J$. In this case, define the 2-form $F(\cdot,\cdot)=g(J\cdot,\cdot)$, then we have that for any $V\in\Gamma T\mathcal{F}$,
 $$i(V)F=0,$$
and
 \begin{eqnarray*}
  [\theta(V)F](s_1,s_2)=[\theta(V)g](Js_1,s_2)=0,~\forall~s_1, s_2\in Q.
 \end{eqnarray*}
Hence, $F$ is a basic 2-form, and $(\mathcal{F},g,J,F)$ is called a transversal almost Hermitian structure.

\section{$C^{\infty}$-pure and full transversal almost complex structure}

For an even codimensional Riemannian foliation $\mathcal{F}$ on $M$ endowed with a transversal almost complex structure $J$ satisfying $\theta(V)J=0$, $\forall~V\in T\mathcal{F}$, denote by $\Lambda^2_B$ the bundle of real basic 2-forms. Since $\theta(V)J=0$, $\forall~V\in T\mathcal{F}$, we have a well-defined action of $J$ on $\Lambda^2_B$ by:
 \begin{eqnarray*}
  J:~\Lambda^2_B&\rightarrow& \Lambda^2_B\\
  \alpha(\cdot,\cdot)&\mapsto& \alpha(J\cdot,J\cdot).
 \end{eqnarray*}

Then by the formula:
 $$
  \alpha(\cdot,\cdot)=\frac{\alpha(\cdot,\cdot)+\alpha(J\cdot,J\cdot)}{2}
   +\frac{\alpha(\cdot,\cdot)-\alpha(J\cdot,J\cdot)}{2},
 $$
we get a splitting
 $$
  \Lambda_B^{2}=\Lambda_{J}^{+} \oplus \Lambda_{J}^{-},
 $$
where $\Lambda_{J}^{+}$ is the bundle of $J$-invariant basic 2-forms, and $\Lambda_{J}^{-}$ is the bundle of $J$-anti-invariant basic 2-forms.

Let $\Omega^2_B$ be the space of basic 2-forms on $M$, $\Omega_J^{+}$ ($\Omega_J^{-}$) the space of $J$-invariant ($J$-anti-invariant) basic 2-forms.

 \begin{defn}
  Let $\mathcal{Z}_B^2$ be the space of basic closed 2-forms on $M$, and let $\mathcal{Z}_{J}^{\pm}=\mathcal{Z}_B^{2} \cap \Omega_{J}^{ \pm}$. Define
   $$
    H_{J}^{ \pm}(\mathcal{F})=\left\{\mathfrak{a} \in H_B^{2}(\mathcal{F}; \mathbb{R}) ~|~ \exists \alpha \in \mathcal{Z}_{J}^{ \pm} \text { such that }[\alpha]=\mathfrak{a}\right\},
   $$
 \end{defn}
and the dimension of $H_{J}^{ \pm}(\mathcal{F})$ are denoted by $h_J^{\pm}$ respectively.

It is obvious that
 $$
  H_{J}^{+}(\mathcal{F})+H_{J}^{-}(\mathcal{F}) \subseteq H_B^{2}(\mathcal{F}; \mathbb{R}).
 $$

 \begin{defn}\label{defn}
  $J$ is said to be $C^{\infty}$-pure is $H_{J}^{+}(\mathcal{F})\cap H_{J}^{-}(\mathcal{F})=0$, and is said to be $C^{\infty}$-full if $H_{J}^{+}(\mathcal{F})+H_{J}^{-}(\mathcal{F})=H_B^{2}(\mathcal{F}; \mathbb{R})$. $J$ is $C^{\infty}$-pure and full if $H_{J}^{+}(\mathcal{F})\oplus H_{J}^{-}(\mathcal{F})=H_B^{2}(\mathcal{F}; \mathbb{R})$.
 \end{defn}

The main result is the following,
 \begin{thm}\label{main}
  Given a codimension four taut Riemannian foliation $\mathcal{F}$ on a closed smooth manifold $M$, if $J$ is a transversal almost complex structure satisfying $\theta(V)J=0$ for any $V\in \Gamma T\mathcal{F}$, then $J$ is $C^{\infty}$-pure and full.
 \end{thm}

 \begin{rem}
  The condition that $\theta(V)J=0$ for any $V\in \Gamma T\mathcal{F}$ seems to be necessary. One of the reason is we need this condition to guarantee $J$ preserves basic 2-forms. The other is that for a taut metric, we can easily construct a $J$ compatible taut metric and the corresponding transversal fundamental 2-form will be a basic form.
 \end{rem}

 \begin{rem}
  For a K-contact manifold $(M,\xi,\eta,\phi,g)$, we have proved that $\phi$ is $C^{\infty}$-pure and full \cite{ZZ}. For the characteristic foliation $\mathcal{F}_{\xi}$, $g$ is taut and $\theta(\xi)\phi=0$, so this can be considered as a special case of Theorem \ref{main}.
 \end{rem}

In order to prove Theorem \ref{main}, we do some preparation. Let $g$ be a bundle-like metric inducing $g_Q$ on $Q$. Define the Hodge star operator:
 $$
  \overline{\ast} : \Omega_{B}^{r}(\mathcal{F}) \rightarrow \Omega_{B}^{q-r}(\mathcal{F})
 $$
as follows:
 $$
  \overline{\ast} \alpha=(-1)^{p(q-r)} \ast(\alpha \wedge \chi_{\mathcal{F}}).
 $$

The relation between $\overline{\ast}$ and the Hodge star operator $\ast$ w.r.t. $g$ is \cite{Ton}
 $$
  \ast\alpha=\overline{\ast} \alpha \wedge \chi_{\mathcal{F}}.
 $$

The scalar product in $\Omega_{B}^{r}(\mathcal{F})$ is defined by
 $$
  \langle\alpha, \beta\rangle_{B}=\int_{M} \alpha \wedge \overline{*} \beta \wedge \chi_{\mathcal{F}},
 $$
which is just the restriction of the usual scalar product on $\Omega^{r}(M)$ to the subspace $\Omega_{B}^{r}(\mathcal{F})$ \cite{Ton}.

Define the formal adjoint $\delta_B:~\Omega_{B}^{r}(\mathcal{F}) \rightarrow \Omega_{B}^{r-1}(\mathcal{F})$ of $d_B=d:~\Omega_{B}^{r-1}(\mathcal{F}) \rightarrow\Omega_{B}^{r}(\mathcal{F})$ by
 $$
  \left\langle d_{B} \alpha, \beta\right\rangle_{B}=\left\langle\alpha, \delta_{B} \beta\right\rangle_{B}.
 $$

It was shown in \cite{KT1, Ton} that, on $\Omega_{B}^{r}(\mathcal{F})$
 $$
  \delta_B=(-1)^{q(r+1)+1} \overline{\ast}\left(d_{B}-\kappa \wedge\right) \overline{\ast}.
 $$

Define the basic Laplacian
 $$
  \Delta_{B}=d_{B} \delta_{B}+\delta_{B} d_{B},
 $$
then the harmonic basic $r$-forms $\mathcal{H}^r_B(\mathcal{F})$ are those satisfying $\Delta_{B} \omega=0$.

We have the following Theorem 7.22 in \cite{Ton}
 \begin{thm}
  Let $\mathcal{F}$ be a transversally oriented Riemannian foliation on a closed manifold $(M,g)$. Assume $g$ to be bundle-like with $\kappa \in \Omega_{B}^{1}(\mathcal{F})$. Then there is a decomposition into mutually orthogonal subspaces
   $$
    \Omega_{B}^{r} \cong \operatorname{im} d_{B} \oplus \operatorname{im} \delta_{B} \oplus \mathcal{H}_{B}^{r}
   $$
  with finite-dimensional $\mathcal{H}^r_B$.
 \end{thm}

 \begin{rem}
  The condition $\kappa \in \Omega_{B}^{1}(\mathcal{F})$ can be removed by the basic decomposition of general mean curvature 1-form, see \cite{Lop}.
 \end{rem}

When the taut foliation $\mathcal{F}$ has codimension $q=4$, we have $\overline{\ast}^2=$id on $\Lambda^{2} Q^{*}$, so we get a decomposition
 $$
  \Lambda^{2} \mathrm{Q}^{*}=\Lambda^{+} \mathrm{Q}^{*} \oplus \Lambda^{-} \mathrm{Q}^{*},
 $$
where $\Lambda^{\pm}$ are the $\pm 1$-eigenspace of $\overline{\ast}$. Suppose $\Omega_B^{\pm}$ are the space of sections of $\Lambda^{\pm} \mathrm{Q}^{*}$, and denote by $\alpha^{+},~\alpha^{-}$ the selfdual, anti-selfdual components of a basic 2-form $\alpha$. Furthermore, we have $\Delta_{B} \overline{\ast}=\overline{\ast} \Delta_{B}$ (note that if $\kappa\not=0$, $\Delta_{B}$ and  $\overline{\ast}$ do not commute). Hence,
 \begin{eqnarray}\label{selfdual}
  H^2(\mathcal{F},\mathbb{R})=\mathcal{H}^2_B(\mathcal{F})=\mathcal{H}^{+}_B(\mathcal{F})\oplus\mathcal{H}^{-}_B(\mathcal{F}),
 \end{eqnarray}
 and we denote the dimension of $\mathcal{H}^2_B(\mathcal{F}),~\mathcal{H}^{+}_B(\mathcal{F}),~\mathcal{H}^{-}_B(\mathcal{F})$ by $b_B^2,~b_B^{+},~b_B^{-}$ respectively.

For a codimension four transversal almost Hermitian manifold $(M,\mathcal{F},g,J,F)$, we have the following relation
 \begin{eqnarray}\label{rel}
  &&\Lambda^{+}_{J}=\mathbb{R}F\oplus\Lambda^{-}_{g_Q}, \Lambda^{+}_{g_Q}=\mathbb{R}F\oplus\Lambda^{-}_{J};\\
  &&\Lambda^{+}_{J}\cap\Lambda^{+}_{g_Q}=\mathbb{R}F, \Lambda^{-}_{J}\cap\Lambda^{-}_{g_Q}=0.\nonumber
 \end{eqnarray}
 Hence, similar to \cite{DLZ}, we have the following two lemmas,
 \begin{lem}\label{Hodge1}
  If $\alpha\in\Omega^{+}_B$ and $\alpha=\alpha_h+d\theta+\delta\Psi$ is its basic Hodge decomposition, then $(d\theta)^{+}_B=(\delta\Psi)^{+}_B$ and $(d\theta)^{-}_B=-(\delta\Psi)^{-}_B$. In particular, the basic 2-form
   $$
    \alpha-2(d\theta)^{+}_B=\alpha_h
   $$
  is harmonic and the 2-form
   $$
    \alpha+2(d\theta)^{-}_B=\alpha_h+2d\theta
   $$
  is closed.
 \end{lem}

 \begin{lem}
  Let $(M^{p+4},\mathcal{F},g,J,F)$ be a closed codimension four taut transversal almost Hermitian manifold. Then $\mathcal{Z}^{-}_{J}\subset\mathcal{H}^{+}_{g_Q}$, and $\mathcal{Z}^{-}_{J}\subset H^{-}_{J}$ is bijective. Furthermore, $H^{-}_{J}=\mathcal{Z}^{-}_{J}=\mathcal{H}^{+,F^{\perp}}_{g_Q}$.
 \end{lem}

With the above preparation, we can present the proof of the main result.\\
\textbf{Proof of Theorem \ref{main}.} Let $g$ be the $J$-compatible metric, and $F$ be the basic 2-form. If $\mathfrak{a}\in H^{+}_J(\mathcal{F})\cap H^{-}_J(\mathcal{F})$, let $\alpha'\in \mathcal{Z}^{+}_J$, $\alpha''\in\mathcal{Z}^{-}_J$ be the representative for $\mathfrak{a}$. Then see page 39 in \cite{Ton},
 $$
  \mathrm{d} \chi_{\mathcal{F}}+\kappa \wedge \chi_{\mathcal{F}}=\varphi_{0} \in F^{2} \Omega^{p+1}.
 $$

Hence, on a codimension four foliation $(M,\mathcal{F})$, for basic 1-form $\gamma$ and basic 2-form $\alpha''$, $\gamma \wedge \alpha^{\prime \prime} \wedge\phi_{0}\in F^{5} \Omega^{p+1}=0$ vanishes. Therefore, by integration by parts, we have
 \begin{eqnarray*}
  0 &=& \int_M \alpha'\wedge\alpha''\wedge\chi_{\mathcal{F}}\\
    &=& \int_M (\alpha''+d_B\gamma)\wedge\alpha''\wedge\chi_{\mathcal{F}}\\
    &=& \int_M \alpha''\wedge\alpha''\wedge\chi_{\mathcal{F}} + \int_M d_B\gamma\wedge\alpha''\wedge\chi_{\mathcal{F}}\\
    &=& \int_M \alpha''\wedge\overline{\ast}\alpha''\wedge\chi_{\mathcal{F}} + \int_M \gamma\wedge d_B\alpha''\wedge\chi_{\mathcal{F}} + \int_M \gamma\wedge\alpha''\wedge \mathrm{d}\chi_{\mathcal{F}}\\
    &=& \int_M |\alpha''|^2_g~\mathrm{d}vol + \int_M \gamma\wedge\alpha''\wedge (\phi_0-\kappa\wedge\chi_{\mathcal{F}})\\
    &=& \int_M |\alpha''|^2_g~\mathrm{d}vol.
 \end{eqnarray*}

Hence, $\alpha''=0$, i.e., $\mathfrak{a}=0$, that's to say $H_{J}^{+}\left(\mathcal{F}\right)\cap H_{J}^{-}\left(\mathcal{F}\right)=0$.

The proof of fullness part is technically almost the same as the proof of Theorem 2.3 in \cite{DLZ}. $\Box$

D. Dom\'{\i}nguez's remarkable theorem \cite{Do} says that for a Riemannian foliation $\mathcal{F}$ on a closed manifold, there always exists a bundle-like metric for $\mathcal{F}$ such that the mean curvature form $\kappa$ is a basic 1-form. F. Kamber and Ph. Tondeur shows $\kappa$ should be closed \cite{KT2}. Furthermore, if $[\kappa]\in H_B^1(\mathcal{F})$ is trivial, then by a suitable conformal change to $g_{T\mathcal{F}}$, the bundle-like metric $g$ can be modified to be a taut metric \cite{KT2}. Since we have an injective map
 $$
  H_B^1(\mathcal{F})\rightarrow H^1(M),
 $$
closed and simply connected Riemannian foliation is always taut \cite{Ton}. Hence, we have the following corollary:

 \begin{cor}
  For a codimension four Riemannian foliation $\mathcal{F}$ on a closed and simply connected smooth manifold $M$, if $J$ is a transversal almost complex structure satisfying $\theta(V)J=0$ for any $V\in \Gamma T\mathcal{F}$, then $J$ is $C^{\infty}$-pure and full.
 \end{cor}

\section{Bounds on $h^{\pm}_J$}
Under the condition of Theorem \ref{main} and by (\ref{selfdual}), we have
 $$
  h^{+}_J+h^{-}_J=b^2_B=b_B^{+}+b_B^{-}.
 $$
Furthermore, by relations (\ref{rel}), the following inequalities holds:
 \begin{eqnarray}\label{inequ}
  h^{+}_J\geq b_B^{-},~ h^{-}_J\leq b_B^{+}.
 \end{eqnarray}
This can be strengthened as follows:

 \begin{lem}\label{sinequ}
  If $(M,\mathcal{F},g,J,F)$ is a closed codimension four almost Hermitian taut Riemannian foliation. Assume that the harmonic part $F_h$ of the transversal Hodge decomposition of $F$ is not identically zero. Then
   $$
    h_J^{+}\geq b_B^{-}+1,~~h_J^{-}\leq b_B^{+}-1.
   $$
 \end{lem}

  \begin{proof}
   Let $F=F_h+\mathrm{d}\theta+\delta\Psi$ be the transversal Hodge decomposition of $F$, then $F+2(\mathrm{d}\theta)^{-}$ is a closed $J$-invariant basic 2-form, and $[F_h+2\mathrm{d}\theta]\in H^{+}_B\cap H^{-}_B$ is nontrivial since $F_h$ is not identically zero.
  \end{proof}

A more specific case is when $F$ is closed, i.e. the manifold $M$ in Lemma \ref{sinequ} is transversal almost K\"{a}hler, we let $\omega=F$.
 \begin{thm}
  If $(M,\mathcal{F},g,J,\omega)$ is taut transversal almost K\"{a}hler of codimension four, then
   $$
    h_J^{+}\geq b_B^{-}+1,~~h_J^{-}\leq b_B^{+}-1.
   $$
 \end{thm}

  \begin{proof}
   Since $g$ is taut, $\overline{\ast}\Delta=\Delta\overline{\ast}$. Hence, $\mathrm{d}\omega=0$ and $\omega\in\Omega^{+}_g$ induces that $\delta_B\omega=0$, i.e., $\omega$ is basic harmonic itself.
  \end{proof}

\vspace{0.5cm}\noindent\textbf{Acknowledgments}.  The author would like to thank professor Hongyu Wang for useful discussion. J.R. Zhou is partially supported by a PRC grant NSFC  11771377 and the Natural Science Foundation of Jiangsu Province(BK20191435).

\end {document}